\documentclass[runningheads]{llncs}
\usepackage[T1]{fontenc}
\usepackage{graphicx}
\usepackage{amsmath}
\usepackage{amssymb}
\usepackage{mathtools}

\spnewtheorem{notation}{Notation}{\bfseries}{}

\newcommand{\N}{\mathbb{N}}
\newcommand{\Z}{\mathbb{Z}}

\newcommand{\set}[2]{\{ #1 \: | \: #2 \}}
\newcommand{\confspace}{\Sigma^{\Z^d}}

\begin{document}
\title{On the surjunctivity and the Garden of Eden theorem for non-uniform cellular automata}
\titlerunning{On surjunctivity and the GoE theorem for non-uniform CA}
%
\author{Katariina Paturi \orcidID{0009-0007-2475-1393} \and Jarkko Kari \orcidID{0000-0003-0670-6138} } 
%
\authorrunning{K. Paturi, J. Kari}
%
\institute{University of Turku, Turku, Finland} 
%
\maketitle              
\begin{abstract}
Non-uniform cellular automata (NUCA) are an extension of cellular automata with multiple local rules in different cells. We show that if the distribution of local rules is uniformly recurrent, or recurrent in the one-dimensional case, the Garden of Eden theorem holds. We also show that for any non-recurrent distribution, there is a substitution of local rules that defines a NUCA which does not satisfy the Garden of Eden theorem. Finally, we show that a rule distribution asymptotic to recurrent distribution defines a surjunctive NUCA.

\keywords{Non-uniform cellular automata  \and Garden of Eden theorem \and Surjunctivity.}
\end{abstract}

\section{Introduction}

The Garden of Eden theorem states that a cellular automaton, or CA for short, is surjective if and only if it is pre-injective. One direction of the Garden of Eden theorem for CA over integer grids was first proven by Moore in 1962 and the other direction by Myhill in 1963. It is known that there are groups that don't satisfy the Garden of Eden theorem, specifically, the theorem is satisfied in amenable groups \cite{cecc,moore,myhill}.

Surjunctivity is a broader notion that was introduced by Gottschalk in 1973. A group is said to be surjunctive if all CA over the group have the property that they are surjective if they are injective. It is known that all sofic groups are surjunctive, but it remains an open problem whether all groups are surjunctive \cite{gottschalk,cecc}.

Finally, non-uniform cellular automata, or NUCA, are a generalization of regular CA which operate with different local rules in different cells. The local rules are given by a special configuration called a local rule distribution. It can easily be seen that the Garden of Eden theorem does not hold in general for non-uniform CA over integer grids. More recently, conditions for the surjunctivity of non-uniform CA have been researched \cite{sipper,dennuz,phung}.

In this work we show that if a NUCA is defined by a uniformly recurrent local distribution, it satisfies the Garden of Eden theorem. More specifically in the 1-dimensional case, we show that the theorem is satisfied if the distribution is recurrent. Conversely, we show that for any one-dimensional non-recurrent configuration, there is an assignment of local rules to symbols such that the obtained rule distribution defines a NUCA which does not satisfy the Moore direction of the Garden of Eden theorem. We also prove the analogous result for the Myhill direction. Finally we show that all distributions asymptotic to recurrent distributions are surjunctive, that is, they define NUCA which are surjective if they are injective.

\section{General definitions}

We define some general notation regarding configurations.

\begin{definition}
Let $\Sigma$ be a finite set called a \emph{state set} and its elements \emph{states}. Let $d\in \N$. A \emph{configuration} is a function $c:\Z^d \rightarrow \Sigma$ and $d$ is the \emph{dimension} of the configuration. An element $\bar x \in \Z^d$ is called a \emph{cell} and $c(\bar x)$ is the \emph{state of cell $\bar x$}.
\end{definition}

\begin{definition}
Let $\Sigma$ be a state-set and $d \in \Z_+$. A finite set $D \subseteq \Z^d$ is called a \emph{finite domain}. A \emph{finite pattern} is a function $p:D\rightarrow \Sigma$.
\end{definition}

\begin{notation}
Let $c \in \Sigma^{\Z^d}$ and $D \subseteq \Z^d$. The  pattern $c_{|D}\in \Sigma^D$ is the unique pattern for which it holds that $c_{|D}(\bar x) = c(\bar x)$ for all $\bar x \in D$.
\end{notation}


\begin{definition} \label{def_copy}
Let $D \subseteq \Z^d$ be a finite domain, $p_1 \in \Sigma^D$ and $\bar r \in \Z^d$. Let $E = \set{\bar x + \bar r}{\bar x \in D}$
and $p_2 \in \Sigma^{E}$. The pattern $p_2$ is a \emph{translated copy of} $p_1$ if 
$p_1(\bar x) = p_2 (\bar x + \bar r)$
for all $\bar x \in D$.
\end{definition}

\begin{notation}
Let $x,y\in \Z$. We denote $[x,y] = \set{z \in \Z}{x \leq z \leq y}$.
\end{notation} 

Recurrence and uniform recurrence are general properties of a dynamical system. The definitions here are equivalent to the usual definitions for configurations under a shift dynamic.

\begin{definition}
A \emph{hypercube} $C \subseteq \Z^d$ is a finite domain such that for some  $w\ in \Z_+$ and cell $\bar{x} \in \Z^d$, 
\begin{align*}
C = \set{\bar x + (v_1,\ldots,v_d)}{\forall i\in [1,d]: 0 \leq v_i < w}.
\end{align*}
The constant $w$ is called the \emph{width} of $C$, and $C$ is called \emph{$w$-wide}.
\end{definition}

\begin{definition}
A configuration $c \in \confspace$ is \emph{recurrent} if for all finite domains $D \subseteq \Z^d$, there is  $E \subseteq \Z^d$ such that $D\neq E$ and $c_{|E}$ is a translated copy of $c_{|D}$.
In other words, a configuration is recurrent if every finite pattern that appears in the configuration, appears at least twice. 

The configuration $c$ is \emph{uniformly recurrent} if for every finite domain $D \subseteq \Z^d$, there exists $w \in \Z_+$ such that for any hypercube of width $w$, the hypercube contains a translated copy of $c_{|D}$. In other words, a configuration is uniformly recurrent if every pattern that appears in it, appears in every hypercube of width $w$.
\end{definition}

\section{Non-uniform cellular automata}

Non-uniform cellular automata are a generalization of cellular automata with the possibility of having multiple different local rules. 

\begin{definition}
Let $\Sigma$ be a state set, $d \in \Z_+$ a dimension and $N = (\bar n_1,\ldots,\bar n_m)$ a tuple of vectors $\bar n_i \in \Z^d$, $1\leq i \leq m$ called a \emph{neighbourhood}. A \emph{local rule} is a function $f:\Sigma^m \rightarrow \Sigma$.
\end{definition}

\begin{notation}Let $N = (\bar n_1,\bar n_2,\ldots,\bar n_m)$ be a neighbourhood, $\bar x \in \Z^d$ and $D \subseteq \Z^d$. The neighbourhood of a cell is denoted as
\begin{align*}
N(\bar x) = \set{\bar x + \bar n}{\bar n \in N}
\end{align*}
and the neighbourhood of $D$ is denoted as
\begin{align*}
N(D) = \set{\bar x + \bar n}{\bar x \in D, \bar n \in N}.
\end{align*}
\end{notation}

\begin{definition} \emph{\cite{dennuz2}}
Let $\mathcal{R}$ be a finite set of local rules with state set $\Sigma$ and neighbourhood $N$. A configuration $\theta \in \mathcal{R}^{\Z^d}$ is called a \emph{local rule distribution}. A \emph{non-uniform cellular automaton} or \emph{NUCA} for short is the tuple $A = (\Sigma, d, N, \mathcal{R}, \theta)$. 
\end{definition} 

\begin{definition}
Let $A = (\Sigma, d, N, \mathcal{R}, \theta)$ be a NUCA where $N =(\bar{n}_1, \ldots, \bar{n}_m)$. The \emph{global update rule} of $A$ is the function $H_\theta:\confspace \rightarrow \confspace$ that maps any configuration $c \in \confspace$ to the configuration $H_\theta (c)$ such that
\begin{align*}
H_\theta(c)(\bar{x}) = \theta(\bar{x})(\bar{x}+ \bar{n}_1, \ldots, \bar{x}+ \bar{n}_m)
\end{align*}
for all $\bar{x} \in \Z^d$.
\end{definition}

If the NUCA is clear from context, it is usually referred to by its global update rule $H_\theta$.

\begin{definition}
Let $A = (\Sigma, d, N, \mathcal{R}, \theta)$ be a NUCA where $N =(\bar{n}_1, \ldots, \bar{n}_m)$. Let $D\subset \Z^d$ be a finite domain. A \emph{update rule over domain $D$} is the function $H_{\theta|D}:\Sigma^{|N(D)|}\rightarrow \Sigma^D$ that maps any finite pattern $p \in \Sigma^{N(D)}$ to the pattern
\begin{align*}
H_{\theta|D}(p)(\bar{x}) = \theta(\bar{x})(p(\bar{x}+\bar{n}_1), \ldots p(\bar{x}+\bar{n}_m))
\end{align*}
for all $\bar{x} \in D$.
\end{definition}

\begin{lemma}  \label{cont}
\emph{\cite{dennuz}} Let $\theta \in \mathcal{R}^{\Z^d}$ be a rule distribution, $A = (\Sigma, d, N, \mathcal{R}, \theta)$ a NUCA. The global update rule $H_\theta:\confspace \rightarrow \confspace$ is continuous.
\end{lemma}

We introduce the concept of local rule templates to characterize rule distributions. A local rule distribution template is a configuration of symbols which function as templates for local rules. Each symbol can then be assigned a local rule, which yields a local rule distribution. This allows for making general statements about distributions which are obtained from different assignments to the same template.

\begin{definition} 
Let $\mathcal{R}$ be a finite set of local rules, $d\in \Z_+$ a dimension and $T$ a finite set whose elements are called \emph{rule templates}. A configuration $\tau \in T^{\Z^d}$ is called a \emph{local rule distribution template} and a function $\alpha:T\rightarrow \mathcal{R}$ is called an \emph{assignment of local rules}. The rule distribution $\tau_\alpha \in \mathcal{R}^{\Z^d}$ such that $\tau_\alpha (x) = \alpha (\tau (x))$ is called $\tau$ \emph{with assignment} $\alpha$.
\end{definition}

\section{The Garden of Eden theorem}

The Garden of Eden theorem, or GoE theorem for short, states that the global rule of a CA is surjective if and only if it is injective on finite configurations. This can be stated more generally by replacing the notion of injectivity on finite configurations with pre-injectivity.

\begin{definition}
Let $c,e \in \Sigma^{\Z^d}$. The \emph{difference set} of $c$ and $e$ is the set
\begin{align*}
\mathrm{diff}(c,e) = \set{\bar x\in \Z^d}{c(\bar x) \neq e(\bar x)}.
\end{align*}
If $\mathrm{diff}(c,e)$ is finite, $c$ and $e$ are called \emph{asymptotic}.
\end{definition}

\begin{definition}
Let $H_\theta$ be the update rule of a NUCA. $H_\theta$ is \emph{pre-injective} if for all asymptotic configurations $c,e \in \confspace$ such that $c \neq e$, it holds that $H_\theta (c) \neq H_\theta (e)$.
\end{definition}

The theorem is known to hold for regular CA \cite{moore,myhill}, but it is also easy to see that it does not hold for NUCA generally. 
In this section we show that the theorem holds for NUCA with uniformly recurrent distributions, and that in the 1-dimensional case, it holds if the distribution is recurrent. 
First we state an auxiliary lemma used in both proofs. It is a simple generalization of a lemma in \cite{dennuz2}.

\begin{lemma} \label{orphan}
Let $\theta \in \mathcal{R}^{\Z^d}$. The $d$-dimensional global update rule $H_\theta$ is surjective if and only if the partial update rule $H_{\theta|C}$ is surjective for all hypercubes $C\subseteq \Z^d$
\end{lemma}

\begin{proof}
Let $N$ be the neighbourhood of $H_\theta$. First assume $H_\theta$ is surjective. Let $C \subseteq \Z^d$ be a hypercube and $p \in \Sigma^C$. Clearly there is a configuration $c \in \confspace$ such that $c_{|C} = p$. Because $H_\theta$ is surjective, there is a configuration $e \in \confspace$ such that $H_\theta (e) = c$. Then $H_{\theta |C} (e_{|N(C)} ) = c_{|C} = p$ and therefore $H_{\theta|C}$ is surjective.

Assume then that $H_{\theta|C}$ is surjective for all hypercubes $C \subseteq \Z^d$. Let $C_i$ be a width $i$ hypercube centered on the origin for all $i \in \Z_+$. Let $c \in \confspace$ and $(e_i)_{i=1}^\infty$ a sequence of configurations $e_i \in \confspace$ such that $H_{\theta|C_i}(e_{i|N(C_i )}) = c_{|C_i}$ for all $i \in \Z_+$. Because $H_{\theta|C_i}$ is surjective, such a sequence exists. 

Let $c_i = H_\theta (e_i)$. Clearly the sequence $(c_i)_{i=1}^\infty$ converges with limit $\lim_{i \rightarrow \infty} c_i =c$. By compactness of the Cantor space, the sequence $(e_i)_{i=1}^\infty$ has a converging subsequence $(e_{i_j})_{j=1}^\infty$. Let $e = \lim_{j \rightarrow \infty} e_{i_j}$. Then by Lemma \ref{cont}, 
\begin{align*}
c = \lim_{i\rightarrow \infty} c_i = \lim_{j \rightarrow \infty} c_{i_j} = \lim_{j \rightarrow \infty} H_\theta (e_{i_j}) = H_\theta (e).
\end{align*}
Therefore $H_\theta$ is surjective.
\qed
\end{proof} 

Additionally, we define the support of a configuration.

\begin{definition}
Let $q \in \Sigma$ be a state. The \emph{$q$-support} of a configuration $c \in  \confspace$ is the set
\begin{align*}
\mathrm{supp}_q(c) = \set{\bar{x} \in \Z^d}{c(\bar{x}) \neq q}.
\end{align*}
\end{definition}

\subsection{Uniformly recurrent distributions}

All NUCA with uniformly recurrent distributions satisfy the GoE theorem. The proof is a modification of the original proof for regular CA found in \cite{moore,myhill}.

\begin{lemma} \label{ineq1} 
\emph{\cite{moore}}
Let $d,s,n,r \in \Z_+$. For all sufficiently large $k \in \Z_+$, it holds that
\begin{align*}
(s^{n^d}-1)^{k^d} < s^{(kn-2r)^d}.
\end{align*}
\end{lemma}

\begin{lemma} \label{goe_unif1}
Let $\theta \in \mathcal{R}^{\Z^d}$ be a uniformly recurrent rule distribution. If $H_\theta$ is not surjective, then it is not pre-injective.
\end{lemma}

\begin{proof}
Let $r \in \Z_+$ be large enough that $H_\theta$ can be defined with radius-$r$ local rules. Let $s = |\Sigma|$. Suppose that $H_\theta$ is not surjective. Then by Lemma \ref{orphan} there's a domain $D \subseteq \Z^d$ such that $H_{\theta|D}$ is not surjective. Because $\theta$ is uniformly recurrent, there is $n \in \Z_+$ such that any $n$-wide hypercube in $\theta$ contains a translated copy of $\theta_{|D}$.

Let $q \in \Sigma$ and $k\in \Z_+$. Let $C\subseteq \Z^d$  be a hypercube of width $kn$ and $C'$ a hypercube of width $kn-2r$ centred on $C$. Let
\begin{align*}
K = \set{c \in \confspace}{\mathrm{supp}_q (c) \subseteq C'}.
\end{align*} 
Now $|K| = s^{|C'|} = s^{(kn-2r)^d}$.

Hypercube $C$ can be partitioned into $k^d$ hypercubes of width $n$, each of which must contain a copy of $\theta_{|D}$ in $\theta$. Consider then the set $H_\theta (K)$. Because $H_{\theta|D}$ is not surjective, there is a finite pattern $p\in \Sigma^D$ with no $H_{\theta|D}$ pre-image. Then for each of the $k^d$ hypercubes in $C$, there is at least one finite pattern with no pre-image. Because they are identical outside of $C$, there are at most $(s^{n^d}-1)^{k^d}$ configurations in $H_\theta (K)$. Now by Lemma \ref{ineq1},
\begin{align*}
|H_\theta(K)| \leq (s^{n^d}-1)^{k^d} < s^{(kn-2r)^d} = |K|
\end{align*}
for sufficiently large $k$. Therefore there must be configurations $c_1,c_2 \in K$ such that $c_1 \neq c_2$ and $H_\theta (c_1) = H_\theta(c_2)$. Hence $H_\theta$ is not pre-injective.
\qed
\end{proof}

\begin{definition}
Let $\bar r \in \Z^d$. The \emph{$\bar r$-shift} $\sigma_{\bar r}: \confspace \rightarrow \confspace$ is the function mapping $c\in \confspace$ such that for all $\bar x \in \Z^d$, 
\begin{align*}
\sigma_{\bar r} (c)(\bar x) = c(\bar x - \bar r).
\end{align*} 
\end{definition}

\begin{lemma} \label{goe_unif2}
Let $\theta \in \mathcal{R}^{\Z^d}$ be a uniformly recurrent rule distribution. If $H_\theta$ is not pre-injective, then it is not surjective.
\end{lemma}

\begin{proof}
Let $r\in 2\Z_+$ be large enough that $H_\theta$ can be defined with radius-$\frac{r}{2}$ local rules. Let $s = |\Sigma|$. Suppose that $H_\theta$ is not pre-injective. Let $c_1,c_2 \in \confspace$ be asymptotic configurations such that $c_1 \neq c_2$ and $H_\theta (c_1) = H_\theta (c_2)$.

Because $c_1$ and $c_2$ are asymptotic, there is a hypercube $D \subseteq \Z^d$ such that $\mathrm{diff}(c_1,c_2)$ is contained within a hypercube centered on $D$ whose width is at least $2r$ less than the width of $D$. Let $n\in \Z_+$ be large enough that every hypercube of width $n$ in $\theta$ must contain a translated copy of $\theta_{|D}$. 

Consider any hypercube $E \subseteq \Z^d$ of width $n$. There is $\bar{y} \in \Z^d$ such that $D' = \set{\bar{x} + \bar{y}}{\bar{x} \in D}  \subseteq D$ and $\theta_{|D'}$ is a copy of $\theta_{|D}$. Let $p_1 = \sigma_{\bar{y}} (c_1)_{|D}$ and $p_2 = \sigma_{\bar{y}} (c_2)_{|D}$. Let then $e_1,e_2 \in \confspace$ be configurations such that $e_{1|D} = p_1$ and $e_{2|D} = p_2$. Consider any cell $\bar{x} \in \Z^d$. If $N(\bar{x}) \cap \mathrm{diff}(e_1,e_2) = \emptyset$, then clearly $H_\theta (e_1) (\bar{x}) = H_\theta (e_2) (\bar{x})$. If $N(\bar{x}) \cap \mathrm{diff}(e_1,e_2) \neq \emptyset$ then $\bar{x}$ is within $\frac{r}{2}$ cells from a differing cell, and
\begin{align*}
H_\theta(e_1)(\bar{x}) = H_\theta (c_1)(\bar{x} - \bar{y}) = H_\theta (c_2)(\bar{x} - \bar{y}) = H_\theta(e_2)(\bar{x}).
\end{align*}
Therefore $H_\theta (e_1) = H_\theta (e_2)$, meaning for any hypercube of width $n$, there are two patterns $p_1$ and $p_2$ which can be replaced with each other without affecting the image.

Let $C$ be a hypercube of width $kn$ for some $k\in \Z_+$, and $C'$ a hypercube of width $kn-2r$ centred on $C$. If $H_\theta$ is surjective, then $H_{\theta|C'}$ is surjective, meaning every pattern in the domain $C'$ has a pre-image in the domain $C$. There are $s^{(kn-2r)^d}$ possible patterns in the domain $C'$. The hypercube $C$ can be partitioned into $k^d$ hypercubes of width $n$, each of which must contain two patterns $p_1$ and $p_2$ such that $p_1$ can be replaced by $p_2$ without affecting the image of $C$. If for each of these hypercubes we fix one of these two patterns to be $p_1$, then every pattern in the domain $C'$ has a pre-image with no $p_2$ pattern in any hypercube. There are at most $(s^{n^d}-1)^{k^d}$ such patterns. 

By Lemma \ref{ineq1}, for sufficiently large $k$, it holds that $(s^{n^d}-1)^{k^d}<s^{(kn-2r)^d}$ and therefore some pattern in the domain $C'$ has no pre-image. Therefore $H_{\theta|C'}$ is not surjective and hence $H_\theta$ is not surjective.
\qed
\end{proof}

\begin{theorem} \label{goe_unif}
Let $\theta \in \mathcal{R}^{\Z^d}$ be a uniformly recurrent rule distribution. The NUCA $H_\theta$ is surjective if and only if it is pre-injective.
\end{theorem}

\begin{proof} 
The statement follows from Lemmas \ref{goe_unif1} and \ref{goe_unif2}. 
\qed
\end{proof}

\subsection{The 1-dimensional case}

In the 1-dimensional case, the distribution only needs to be recurrent for the GoE theorem to hold. The proof is again a modification of the original proof for regular CA.

\begin{lemma} \label{ineq2}
Let $s,n,r \in \Z_+$. For all sufficiently large $m \in \Z_+$, it holds that
\begin{align*}
(s^n-1)^m s^{k-nm} < s^{k-2r},
\end{align*}
for all $k \in \Z_+$.
\end{lemma}

\begin{proof}
By Lemma \ref{ineq1}, for all sufficiently large $m \in \Z_+$ it holds that
\begin{align*}
(s^n -1)^m < s^{mn-2r}. 
\end{align*}
Then for all $k \in \Z_+$
\begin{align*}
(s^n -1)^m s^{k-nm} &< s^{mn-2r} s^{k-nm} = s^{k-2r}
\end{align*}
\qed
\end{proof}

\begin{lemma} \label{rgoe1}
Let $\theta \in \mathcal{R}^\Z$ be a recurrent rule distribution. If $H_\theta$ is not surjective, then it is not pre-injective.
\end{lemma}

\begin{proof}

Let $r \in \Z_+$ be large enough that $H_\theta$ can be defined with radius-$r$ local rules. Let $s = |\Sigma|$. Suppose that $H_\theta$ is not surjective. Then by Lemma \ref{orphan} there are $i,j \in \Z$ such that $H_{\theta|[i,j]}$ is not surjective. Let $n = |j-i| +1$.

Let $m \in \Z_+$ be large enough that $(s^n-1)^m s^{k-nm} < s^{k-2r}$ for all $k \in \Z_+$. By Lemma \ref{ineq2}, such a number $m$ exists. Because $\theta$ is recurrent, it has infinitely many copies of the pattern $\theta_{|[i,j]}$. Let $C \subseteq \Z^d$ be a segment such that $\theta_{|C}$ contains $m$ disjoint copies of $\theta_{|[i,j]}$. Let $k$ be the length of $C$, and let $C'$ be a length $k-2r$ segment centered on $C$. Finally, let $q\in \Sigma$ and
\begin{align*}
K = \set{c \in \Sigma^\Z}{\mathrm{supp}_q (c) \subseteq C'}.
\end{align*}
Now $|K| = s^{|C'|} = s^{k-2r}$. Clearly any pair of configurations from $K$ is asymptotic.

Consider the image $H_\theta (K)$. Because $H_{\theta|[i,j]}$ is not surjective, there is a finite pattern $p \in \Sigma^{[i,j]}$ with no $H_{\theta|[i,j]}$ pre-image. Then no translated copy can appear in the same position as a copy of $\theta_{|[i,j]}$ in any configuration in $H_\theta (K)$. Then there are at most $(s^n-1)^ms^{k-nm}$ configurations in $H_\theta (K)$. 

Therefore, due to the selection of $m$
\begin{align*}
|H_\theta (K)| \leq (s^n -1)^m s^{k-nm} < s^{k-2r} = |K|.
\end{align*}
Then there must be configurations $c_1, c_2 \in K$ such that $c_1 \neq c_2$ and $H_\theta (c_1) = H_\theta (c_2)$. Hence $H_\theta$ is not pre-injective.

\end{proof}

\begin{lemma} \label{rgoe2}
Let $\theta \in \mathcal{R}^\Z$ be a recurrent rule distribution. If $H_\theta$ is not pre-injective, then it is not surjective. 
\end{lemma}

\begin{proof}
Let $r \in 2\Z_+$ be large enough that $H_\theta$ can be defined with radius-$\frac{r}{2}$ local rules. Let $s = |\Sigma|$. Suppose that $H_\theta$ is not pre-injective. Let $c_1,c_2 \in \Sigma^\Z$ be asymptotic configurations such that $c_1 \neq c_2$ and $H_\theta (c_1) = H_\theta (c_2)$.

Because $c_1$ and $c_2$ are asymptotic, there is a segment $[i+r,j-r] \subseteq \Z$ such that $\mathrm{diff}(c_1,c_2) \subseteq [i+r, j-r]$. Let $n = |i-j| +1$. Let $p_1 = c_{1|[i,j]}$ and $p_2 = c_{2|[i,j]}$. Let $m$ be as in Lemma \ref{ineq2}.

Let $e_1 \in \Sigma^\Z$ be a configuration containing a copy of $p_1$ in some segment $[i+y,j+y]$ where $y\in\Z$ is such that $\theta_{|[i+y,j+y]}$ is a copy of $\theta_{|[i,j]}$. Let $e_2 \in \Sigma^\Z$ be the configuration obtained by replacing the translated copy of $p_1$ in $e_1$ with a translated copy of $p_2$. Consider cell $x \in \Z$. If $N(x) \cap \mathrm{diff}(e_1,e_2) = \emptyset$, then clearly $H_\theta (e_1)(x) = H_\theta (e_2)(x)$.  If $N(x) \cap \mathrm{diff}(e_1,e_2) \neq \emptyset$ then the distance of $x$ from a differing cell is at most $\frac{r}{2}$. Then $x \in [i+y+\frac{r}{2},j+y-\frac{r}{2}]$, and
\begin{align*}
H_\theta (e_1)(x) = H_\theta (c_1)(x-y)= H_\theta(c_2)(x-y) = H_\theta (e_2)(x).
\end{align*}
Therefore $H_\theta (e_1) = H_\theta (e_2)$, meaning a copy of $p_1$ can be replaced with a copy of $p_2$ without affecting the image of the configurations, provided that the copy of $p_1$ lies in the same cells as a copy of $\theta_{|[i,j]}$ in $\theta$.

Let $C$ be a segment of such that $\theta_C$ contains $m$ disjoint copies of $\theta_{[i,j]}$. Since $\theta$ is recurrent, such a segment exists. Let $k\in \Z_+$ be the length of $C$. Let $C'$ be a segment of length $kn-2r$ centered on $C$. If $H_\theta$ is surjective, then $H_{\theta|C'}$ is surjective, meaning every pattern in the domain $C'$ has a pre-image in the domain $C$. 

There are $s^{k-2r}$ possible patterns in the domain $C'$. On the other hand, because each copy of pattern $p_1$ that shares its domain with one of the $m$ disjoint copies of $\theta_{|[i,j]}$ can be replaced with a copy of $p_2$ without affecting the image, each pattern in $C'$ has a pre-image with no copies of $p_1$ on a copy of these $m$ disjoint positions. There are $(s^n-1)^m s^{n(k-m)}$ such patterns. Because $(s^n-1)^m s^{k-nm} < s^{k-2r}$,  there is some pattern in $C'$ with no pre-image. Therefore $H_{\theta|C'}$ is not surjective and hence $H_\theta$ is not surjective.
\qed
\end{proof}

\begin{theorem}
Let $\theta \in \mathcal{R}^\Z$ be a recurrent rule distribution. The NUCA $H_\theta$ is surjective if and only if it is pre-injective.
\end{theorem}

\begin{proof}
The statement follows from Lemmas \ref{rgoe1} and \ref{rgoe2}.
\qed
\end{proof}

\section{Non-recurrent rule distributions}

We now know that a 1-dimensional NUCA given by a recurrent rule distribution must satisfy the GoE theorem. Then, given a recurrent distribution template, any assignment to that template satisfies the theorem. In this section we show that these are in fact exactly the templates that always satisfy the theorem; for any non-recurrent template, there exists an assignment which is pre-injective, but not surjective, and an assignment which is surjective, but not pre-injective.

We begin by examining the special case of a distribution of two local rule templates, where one of the two templates can appears $n$ times in a row at most once. We show that for such a template, for any $n\geq 1$, for both directions of the GoE theorem, there is an assignment that does not satisfy the theorem.

\subsection{Moore's theorem}

In this section we present a construction for assignments to the aforementioned special case templates which give distributions whose associated global update rules are pre-injective, but not surjective. First we define some notation and terminology used in the following proofs.

\begin{notation}
For $a,b \in \Z_2$, we denote $a \oplus b = a + b \mod 2$. For $a_1,\ldots,a_n \in \Z_2$, we denote
\begin{align*}
\bigoplus_{i=1}^{n} a_i = a_1 \oplus \ldots \oplus a_n.
\end{align*}
\end{notation} 

\begin{notation}
Let $n \geq 1$ and $\Sigma = \Z_2$. The local rules $f_n$ and $g_n$ are ones with neighbourhood $N =(-n,\ldots,n)$ and which map
\begin{align*}
f_n (a_{-n} ,\ldots, a_n) &= a_{-n} \oplus a_0, \\
g_n (a_{-n} ,\ldots, a_n) &= a_{-n+1} \oplus \ldots \oplus a_n.
\end{align*}
An illustration of these rules is in Figure \ref{fig_fg}.

Let $\theta \in \{f_n,g_n\}^\Z$. For each cell $x \in \Z$, if $\theta (x) = f_n$ we call $x$ an \emph{$f$-cell} and if $\theta (x) = g_n$ we call it a \emph{$g$-cell}. We say that cell $x$ \emph{sees} cell $y$ if $y \in N(x)$ and the state of $y$ affects the computation of the local rule at cell $x$. In these terms, an $f$-cell $x$ sees cells $x-n$ and $x$, and a $g$-cell sees all cells in the range $x-n+1,\ldots,x+n$.

For the sake of clarity, we omit the states of dummy neighbours, and denote
\begin{align*}
f (a_{-n}, a_0) &= f_n(a_{-n} ,\ldots, a_n), \\
g (a_{-n+1} ,\ldots, a_n) &= g_n (a_{-n} ,\ldots, a_n),
\end{align*}
when $n$ is fixed and clear from context.
\end{notation}

\begin{figure}
\begin{center}
\includegraphics[scale=0.8]{"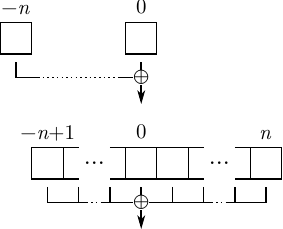"}
\caption{The operation of rules $f_n$ and $g_n$, top to bottom respectively.} \label{fig_fg}
\end{center}
\end{figure}

\begin{lemma} \label{nsurj}
Let $n \geq 1$ and $\Sigma = \Z_2$. Let $\mathcal{R} = \{f_n, g_n\}$ and $\theta \in \mathcal{R}^\Z$ be a rule distribution with pattern $g_n(f_n)^n$. Then there is a configuration $c\in \Sigma^\Z$ with no $H_\theta$ -pre-image.
\end{lemma}

\begin{proof}
Assume the $n$-long block of $f$-cells is in the segment $[x+1,x+n]$, meaning $\theta (x) = g_n$.  Let $c\in \Sigma^\Z$ be a configuration with $c(x) = 1$ and $c(y) = 0$ for all $y\in [x+1,x+n]$. This is illustrated in Figure \ref{fig_nsurj}.

\begin{figure}
\begin{center}
\includegraphics[scale=0.8]{"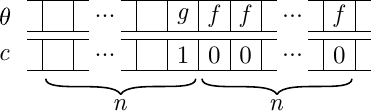"}
\caption{The rule distribution and configuration used in the argument for non-surjectivity of $H_\theta$. Pictured is the range seen by the noted cells.} \label{fig_nsurj}
\end{center}
\end{figure}

Now suppose there is $e \in  \Sigma^\Z$ such that $H_\theta (e) = c$. Because $c(x) = 1$, $x$ sees an odd number of cells with state 1 in $e$. Therefore there are an odd number of cells with state 1 in $e$ in the range $[x-n+1, x+n]$. 

Then consider the $f$-cells. Each $y \in [x+1,x+n]$ sees $y-n$ and $y$. Then the $f$-cells in total see the cells in the ranges $[x-n+1,x-n]$ and $[x+1, x+n]$, covering the full range $[x-n+1, x+n]$ seen by $x$. Additionally, no two cells in $[x+1,x+n]$ see the same cell. Then because $c(y) = 0$ for all $y \in [x+1,x+n]$, each $y$ sees an even number of cells with state 1. This then implies there are an even number of cells with state 1 in the range $[x-n+1,x-n]$, which is a contradiction. Hence $c$ has no pre-image.
\qed
\end{proof}

\begin{lemma} \label{preinj}
Let $n \geq 1$ and $\Sigma = \Z_2$. Let $\mathcal{R} = \{f_n, g_n\}$ and $\theta \in \mathcal{R}^\Z$ be such that there is at most one length $n$ block of adjacent cells with rule $f_n$ (and hence all other blocks of adjacent rules $f_n$ are at most $n-1$ cells long). Then $H_\theta$ is pre-injective.
\end{lemma}

\begin{proof}
Suppose $c_1,c_2 \in \Sigma^\Z$ are asymptotic and $c_1 \neq c_2$. Let $e \in \Sigma^\Z$ be the configuration with $e(x) = c_1(x) \oplus c_2(x)$. Then $e(x) = 1$ if and only if $c_1(x) \neq c_2(x)$. Assume that $H_\theta (c_1) = H_\theta (c_2)$. 

Notice that $H_\theta (c_1)(x) = H_\theta (c_2)(x)$ if and only if the number of cells $y$ seen by $x$ such that $c_1(y) \neq c_2 (y)$ is even. This is if and only if $x$ sees an even number of cells $y$ where $e(y) = 1$, which in turn is if and only if $H_\theta (e)(x) = 0$. Hence $H_\theta (c_1) = H_\theta (c_2)$ if and only if $H_\theta (e)(x) = 0$ for all $x \in \Z$.

Let $x\in \Z$ be the rightmost cell such that $e(x) = 1$. Such a cell exists, because $c_1$ and $c_2$ are asymptotic. If then $\theta (x+n-1) = g_n$, then $H_\theta (e)(x+n-1) = g(1,0,\ldots,0) = 1$, which is a contradiction. Therefore $\theta(x+n-1) = f_n$.  Let us prove next that $e(x-n) = 1$ and, when $n\geq 2$, $e(x-n+1) = \ldots = e(x-1) = 0$. If $n=1$, then because $f(e(x-1), e(x)) = f(e(x-1),1) = 0$, we have $e(x-1) = 1$.

If $n\geq 2$ then $e(x+n-1)=0$, and because $f(e(x-1), e(x+n-1)) = f(e(x-1),0) = 0$, we have $e(x-1) = 0$. Suppose then for some $y \in [x+1, x+n-1]$ and all $z \in [y,x+n-1]$ it holds that $\theta(z) = f_n$. Then also $e(z-n) = 0$ for each $z$ in the range. Now $\theta (y-1) = f_n$; if $\theta (y-1) = g_n$, then $y-1$ would see exactly one cell with state $1$, meaning $H_\theta(e)(y-1) = 1$. Then by induction, for all $y \in [x, x+n-1]$ it holds that $\theta (y) = f_n$. Also then it holds that $e(x-n) = 1$ and $e(x-n+1) = \ldots = e(x-1) = 0$. 

Note that there now is a length $n$ block of $f$-cells in the range $[x, x+n-1]$, meaning that to the left of cell $x$ there can only be at most length $n-1$ blocks of $f$-cells. The obtained $\theta$ and $e$ are illustrated in Figure \ref{fig_inj1}.

\begin{figure}
\begin{center}
\includegraphics[scale=0.8]{"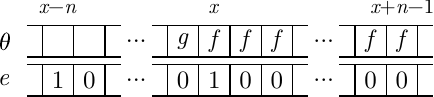"}
\caption{The obtained $\theta$ and $e$. Indices of the cells are annotated on the top row. } \label{fig_inj1}
\end{center}
\end{figure}

Suppose then that somewhere in $e$ there are two copies of the same $n$-long segment back to back. That is, somewhere in $e$ there appears the pattern $a_1 a_2 \ldots a_n a_1 a_2 \ldots a_n$. Let $y \in \Z$ be the leftmost cell of this pattern. Suppose also that the $n$-block of $f$-cells is somewhere to the right of cell $y+n-1$. We show that then $e(y-1) = a_n$.

Suppose $e(y-1) = \bar a_n \neq a_n$. Assume first that $n=1$. Since the only $f$-cell is to the right of $y$, for all $z<y$, $\theta (z) = g_1$. Then $H_\theta (e) (y-1) = g(e(y-1),e(y)) = g(\bar a_1,a_1) = 1$ which is a contradiction. Hence $e(y-1) = a_1$.

Assume then that $n\geq 2$. Now $\theta(y+n-1) = g_n$; if $\theta(y+n-1) = f_n$ then $H_\theta (e)(y+n-1) = f(\bar a_n, a_n) = 1$, a contradiction. Next we show that for all $z \in [y-1, y+n-2]$,  $\theta(z) = f_n$. Let $k\in [0,n-1]$, and assume that for all $z \in [y+k, y+n-2]$ it holds that $\theta (z) = f_n$. Then also it holds that for all such $z$, that $e(z) = e(z-n)$. Therefore $\theta(y+k-1) = f_n$; if $\theta (y+k-1) = g_n$, then

\begin{align*}
H_\theta (e) (y+k-1) &= 
g(a_{k+1},\ldots, a_{n-1}, \bar a_n, a_1, \ldots, a_{n-1}, a_n, a_1, \ldots, a_{k})
\\
&= a_1\oplus \ldots \oplus a_{n-1} \oplus \bar a_n \oplus a_1 \oplus \ldots \oplus a_{n-1} \oplus a_n
\\
&= (a_1 \oplus a_1) \oplus \ldots  \oplus (a_{n-1} \oplus a_{n-1}) \oplus (\bar a_n \oplus a_n)
\\
&= 0 \oplus \ldots \oplus 0 \oplus \ldots \oplus 0 \oplus 1 = 1.
\end{align*}


Now by induction, for any $n\geq 2$ and all $z \in [y-1, y+n-1]$, $\theta(z) = f_n$. But this is a contradiction, because now we have a length $n$ block of rules $f_n$ to the left of cell $y+n-1$ and hence $e(y-1) = a_n$. This argument is illustrated in Figure \ref{fig_inj2}.

\begin{figure}
\begin{center}
\includegraphics[scale=0.9]{"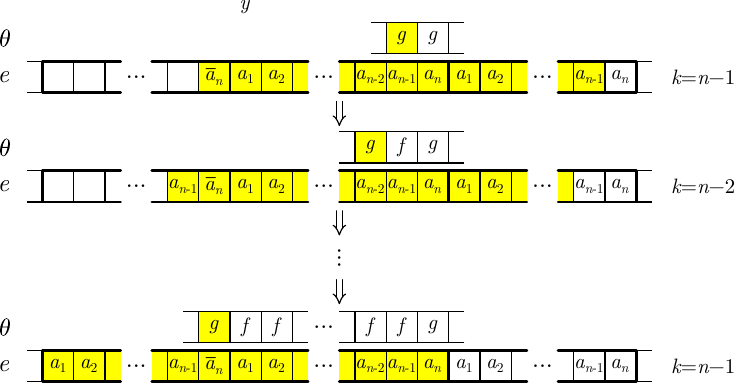"}
\caption{Illustration of the argument for why two repeated $n$-long patterns force a new repetition to their left. The position of cell $y$ is annotated on the top row. The highlighted cells are seen by the highlighted rule $g_n$.} \label{fig_inj2}
\end{center}
\end{figure}

Next, we recursively apply the previous argument to show that there is a copy of the pattern $a_1\ldots a_n$ immediately to the left of the original pair of identical patterns. Suppose we know that immediately left of the cell $y$ is the pattern $a_i\ldots a_n$, where $i \in [2,n]$. Then we have the pattern $a_i\ldots a_n a_1 \ldots a_{i-1} a_i\ldots a_n a_1 \ldots a_{i-1}$ in $e$ and the block of $f$-cells still to the right of its center, meaning by the previous, there is a cell with state $a_{i-1}$ immediately to the left of this pattern. Then by induction, the statement follows.

Now by the preceding, there must be infinitely many copies of $a_1\ldots a_n$ to the left in $e$. Then by the first part of the proof, there are infinitely many copies of the pattern $1(0^{n-1})$ in $e$. This is a contradiction, because it would imply there are infinitely many cells $z$ where $c_1(z) \neq c_2(z)$, but $c_1$ and $c_2$ were assumed to be asymptotic. Hence $H_\theta(c_1) \neq H_\theta (c_2)$ and $H_\theta$ is pre-injective. 
\qed
\end{proof}

\subsection{Myhill's theorem}

In this section we present a construction for assignments to the special case templates which give distributions whose associated global update rules are surjective, but not pre-injective. First we define some notation used in the following proofs.

\begin{notation}
Let $n\geq 1$ and $\Sigma = \Z_2 \times \Z_2$. For any $a=(a_1,a_2) \in \Sigma$, we denote $a^{(1)} = a_1$ and $a^{(2)} = a_2$.
\end{notation}

\begin{definition}
The local rules $\gamma_n$ and $\delta_n$ are ones with neighbourhood $(0,\ldots,n+1)$ such that
\begin{align*}
\gamma_n (a_0, \ldots, a_{n+1})^{(1)} &= a_0^{(1)} \oplus a_{n+1}^{(1)} \oplus a_0^{(2)}, \\
\gamma_n (a_0, \ldots, a_{n+1})^{(2)} &= \delta_n (a_0, \ldots, a_{n+1})^{(2)} = \bigoplus_{i=1}^{n+1} a_i, \\
\delta_n (a_0, \ldots, a_{n+1})^{(2)} &= \bigoplus_{i=1}^{n} a_i,
\end{align*}
for all $a_0,\ldots,a_{n+1} \in \Sigma$. These rules are illustrated in Figure \ref{fig_fg_myhill}. Like previously, if for $x\in \Z$ it holds that $\theta (x) = \gamma_n$, we call $x$ a $\gamma$-\emph{cell} and if $ \theta (x) = \delta_n$, we call $x$ a $\delta$-\emph{cell}. 
\end{definition}

\begin{figure}
\begin{center}
\includegraphics[scale=0.8]{"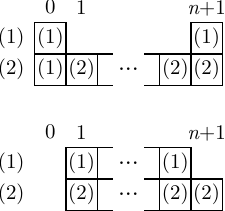"}
\caption{Operation of rules $\gamma_n$ and $\delta_n$, top to bottom respectively. The cells marked $(1)$ are added together on the first track and the ones marked $(2)$ on the second track.} \label{fig_fg_myhill} 
\end{center}
\end{figure}

\begin{lemma} \label{npreinj}
Let $n\geq 1$ and $\Sigma = \Z_2 \times \Z_2$. Let $\mathcal{R} = \{\gamma_n, \delta_n \}$ and $\theta \in \mathcal{R}^\Z$ a rule distribution with pattern $\delta_n(\gamma_n)^n \delta_n$. Then there are asymptotic configurations $c,e\in \Sigma^\Z$ such that $c \neq e$ and $H_\theta (c) = H_\theta (e)$.
\end{lemma}

\begin{proof}
Assume the length $n$ block of $\gamma$-cells is in segment $[x,x+n-1]$. Let $c,e\in \Sigma^\Z$ be such that $c(y) = (0,0)$ for all $y \in \Z$ and $e(x+n) = (1,0)$, $e(y') = (0,0)$ for all $y'\in \Z$, $y'\neq x+n$. Clearly $c$ and $e$ are asymptotic and $c\neq e$ and since the local rules in $\theta$ only add modulo 2, $H_\theta (c) = c$. 

Next we show that $H_\theta (e) = c$. For any cell $y \in \Z$ that doesn't see cell $x+n$, clearly $H_\theta (e) (y) = (0,0)$. Consider then cells that see $x+n$. Because both $\gamma$- and $\delta$-cells only see $n+1$ cells to their right and none to their left, the only cells that can see $x+n$ are in the segment $[x-1, x+n]$. For any $y \in [x,x+n-1]$,
\begin{align*}
H_\theta (e)(y)^{(1)} &= e(y)^{(1)} \oplus e(y+n)^{(1)} \oplus e(y)^{(2)} \\
&= 0 \oplus 0 \oplus 0 = 0, \\
H_\theta (e)(y)^{(2)} &= \bigoplus_{i=1}^{n+1} e(y+i)^{(2)} = 0 \oplus \ldots \oplus 0 = 0,
\end{align*}
Because none of the cells sees the first track component of $x+n$. For cells $x-1$ and $x+n$, we have
\begin{align*}
H_\theta (e)(x-1)^{(1)} &= \bigoplus_{i=1}^{n} e(x-1+i)^{(1)} = 0 \oplus \ldots \oplus 0 = 0 \\
&=  \bigoplus_{i=1}^{n} e(x+n+i)^{(1)} = H_\theta (e)(x+n)^{(1)},\\
H_\theta (e)(x-1)^{(2)} &= \bigoplus_{i=1}^{n+1} e(x-1+i)^{(2)} = 0 \oplus \ldots \oplus 0 = 0 \\
&=  \bigoplus_{i=1}^{n+1} e(x+n+i)^{(2)} = H_\theta (e)(x+n)^{(2)},
\end{align*}
because neither sees the first track component of $x+n$ either. Hence for all $y\in [x-1, x+n]$, $H_\theta(y) = (0,0).$ Therefore $H_\theta (e) = c = H_\theta (c)$. This argument is illustrated in Figure \ref{fig_npreinj}.
\qed

\begin{figure}
\begin{center}
\includegraphics[scale=0.8]{"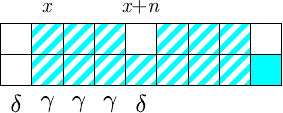"}
\caption{The cells seen by cells in the pattern $\delta_n(\gamma_n)^n \delta_n$ in the case where $n=3$. Cells seen by only $\gamma$-cells are highlighted in solid colour and cells seen by both $\delta$- and $\gamma$-cells are highlighted in striped colour.} \label{fig_npreinj}
\end{center}
\end{figure}

\end{proof}

\begin{lemma} \label{surj}
Let $n\geq 1$ and $\Sigma = \Z_2 \times \Z_2$. Let $\mathcal{R} = \{\gamma_n, \delta_n \}$ and $\theta \in \mathcal{R}^\Z$ a rule distribution with at most one length $n$ block of adjacent $\gamma$-cells. The global update rule $H_\theta$ is surjective.
\end{lemma}

\begin{proof}
Let $c \in \Sigma^\Z$. We show that there is $e\in \Sigma^\Z$ such that $H_\theta (e) = c$. Let $k\leq n$ be the length of the longest block of $\gamma$-cells in $\theta$. If there are no $\gamma$-cells, it is easy to see that $H_\theta$ is surjective, 
so assume $k \geq 1$. For the sake of clarity we assume a block of $\gamma$-cells is in the segment $[1,k]$. This does not effect the result because the defined configuration can be shifted to match the distribution. Then let
\begin{align*}
e(1)^{(1)} &= e(2)^{(1)} = \ldots = e(n-1)^{(1)} = e(n+1)^{(1)} \\
&= e(2)^{(2)} = \ldots = e(n-1)^{(2)} = 0,
\end{align*}
and 
\begin{align*}
&e(n)^{(1)} = c(0)^{(1)},\\ 
&e(n+i+1)^{(1)} = c(i)^{(1)}, \text{ when } i\in \{2,\ldots, k-1 \},\\
&e(n+k+1)^{(1)} = e(k)^{(1)} \oplus c(k)^{(1)}\\
&e(n+2)^{(1)} = \begin{cases}
&c(k+1)^{(1)} \oplus (\bigoplus_{j=n+3}^{n+k+1} e(j)^{(1)}), \text{ if } k\geq n-1 \\
&c(k+1)^{(1)} \oplus e(n)^{(1)} \oplus (\bigoplus_{j=n+3}^{n+k+1} e(j)^{(1)}),\text{ otherwise,}
\end{cases}\\
&e(1)^{(2)} = e(n+2)^{(1)} \oplus c(1)^{(1)}, \\
&e(n+i)^{(2)} =(\bigoplus_{j=i+1}^{n+i} e(j)^{(2)}) \oplus c(i)^{(2)}, \text{ when } i \in \{1, \ldots k+2\}. \\
\end{align*}
This construction is pictured in Figures \ref{fig_surj4} and \ref{fig_surj1}.

\begin{figure}
\begin{center}
\includegraphics[scale=0.8]{"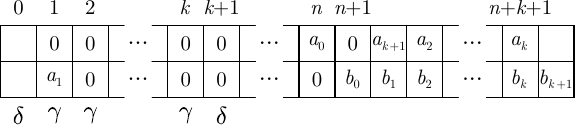"}
\caption{Beginning the construction of the pre-image. The indices of cells are annotated above. Cells marked $a_i$ and $b_i$ are cells reserved for fixing the state of $H_\theta (i)$ on track 1 and 2 respectively. The rules assigned to each cell are annotated on the bottom.} \label{fig_surj4}
\end{center}
\end{figure}

\begin{figure}
\begin{center}
\includegraphics[scale=0.8]{"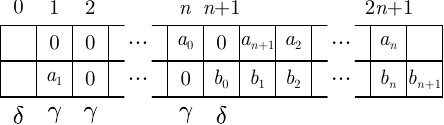"}
\caption{Beginning the construction of the pre-image in the case that $k=n$.} \label{fig_surj1}
\end{center}
\end{figure}

Now $H_\theta (e) (x) = c(x)$ when $x \in [0,k+1]$. For each cell $x^{(i)}$ there's a reserved fixing cell whose state is the sum of $c(x)^{(i)}$ and the states of all the other cells seen by $x^{(i)}$, modulo 2. Then because $H_\theta (e)(x)^{(i)}$ is the modulo 2 sum of the cells seen by $c(x)^{(i)}$, clearly $H_\theta (e)(x)^{(i)} = c(x)^{(i)}$. These states are also well defined, because their definitions only reference already defined states.

Next we recursively define the rest of configuration $e$. First, let $x > k+1$ and assume that $e(x)^{(2)}, e(x+1)^{(2)}, \ldots, e(x+n)^{(2)}$ have already been defined. Then let 
\begin{align*}
e(x+n+1)^{(2)} &= (\bigoplus_{j=1}^{n} e(x+j)^{(2)}) \oplus c(x)^{(2)},
\end{align*}
meaning $H_\theta (e)(x)^{(2)} = c(x)^{(2)}$. Now $e(x)^{(2)}$ is inductively defined for all $x > k+1$. 


Assume then that $e(x)^{(1)}, e(x+1)^{(1)}, \ldots, e(x+n-1)^{(1)}$ have already been defined. If $\theta (x) = \delta_n$ and $\theta(x-1) = \delta_n$, let
\begin{align*}
e(x+n)^{(1)} &= (\bigoplus_{j=1}^{n-1} e(x+j)^{(1)}) \oplus c(x)^{(1)},
\end{align*}
meaning $H_\theta (e)(x) = c(x)$. If $\theta (x) = \gamma_n$ and $\theta(x-1) = \delta_n$, then $x$ is the leftmost cell in a length $m$ block of $\gamma$-cells, where $1\leq m \leq k$. Then let
\begin{align*}
e(x+i+n+1)^{(1)} = \: &e(x+i)^{(1)} \oplus e(x+i)^{(2)} \oplus c(x+i)^{(1)} \\
&\text{when } i \in \{0,\ldots, m-1 \}, \\
e(x+n)^{(1)} = \: &(\bigoplus_{j=m+1}^{n-1} e(x+j)^{(1)}) \\
& \oplus (\bigoplus_{j=n+2}^{n+m} e(x+j)^{(1)}) \oplus c(x+m)^{(1)}.
\end{align*}
 This is illustrated in Figure \ref{fig_surj2}. These states are indeed well defined. The states $e(x), \ldots, e(x+m-1)$ are defined by assumption. Using this, states $e(x+n-1), \ldots, e(x+n+m)$ can be defined recursively. Then every cell seen by $x+m$ except the first track component of $x+n-1$ has already been defined, because the furthest right cell $x+m$ can see the first track component of is $x+n+m$. 

Now $H_\theta (e) (x+i)^{(1)} = c(x+i)^{(1)}$ for all $i \in \{0,\ldots, m \}$. Note that this also covers all cases in which $\theta(x-1) = \gamma_n$, because the entire block of $\gamma$-cells was handled at once. Then by induction $H_\theta (e)(x) = c(x)$ for all $x>k+1$.


\begin{figure}
\begin{center}
\includegraphics[scale=0.8]{"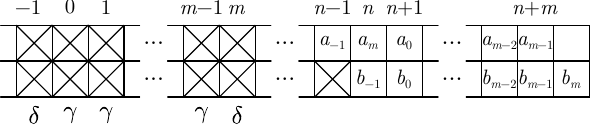"}
\caption{Construction of the pre-image going to the right. The relative indices of cells are annotated above. Cells marked $a_i$ and $b_i$ are cells reserved for fixing the state of $H_\theta (i)$ on track 1 and 2 respectively. Crossed out cells are cells whose state has already been defined in $e$.} \label{fig_surj2}
\end{center}
\end{figure}

Finally, let $x < 0$. First assume that $e(x+2)^{(2)}, e(x+3)^{(2)}, \ldots, e(x+n+1)^{(2)}$ have already been defined. Let  
\begin{align*}
e(x+1)^{(2)} = c(x)^{(2)} \oplus (\bigoplus_{j=2}^{n+1} e(x+j)^{(2)}),
\end{align*}
meaning $H_\theta (e) (x)^{(2)} = c(x)^{(2)}$. Now $e(x)^{(2)}$ is inductively defined for all $x < 0$. Assume then that $e(x+2)^{(1)}, e(x+3)^{(1)}, \ldots, e(x+n+1)^{(1)}$ have been defined. If $\theta(x) = \delta_n$ and $\theta (x+1) = \delta_n$, let
\begin{align*}
e(x+1)^{(1)} = c(x)^{(1)} \oplus (\bigoplus_{j=2}^{n} e(x+j)^{(1)}),
\end{align*}
meaning $H_\theta (e) (x)^{(1)} = c(x)^{(1)}$. If $\theta (x) = \gamma_n$ and $\theta (x+1) = \delta_n$, then $x$ is the rightmost cell in a length $m$ block of $\gamma$-cells, where $1 \leq m \leq k$. Then let
\begin{align*}
e(x-i)^{(1)} =\: &c(x-i)^{(1)} \oplus e(x-i+n+1)^{(1)} \oplus e(x-i)^{(2)} \\
&\text{when } i \in \{0, \ldots, m-1 \}, \\
e(x+1)^{(1)} =\: &(\bigoplus_{j=0}^{m-1} e(x-i)^{(1)})  \\
&\oplus (\bigoplus_{j=2}^{n-m} e(x+i)^{(1)}) \oplus c(x-m)^{(1)}.
\end{align*} 
This is illustrated in \ref{fig_surj3}. These states are well defined: for each $i \in \{0, \ldots, m-1 \}$, $e(x-i)^{(2)}$ was defined earlier and $e(x-i+n+1)^{(1)}$ is defined by assumption since $m<n$. Then every cell seen by $x-m$ has been defined except for the first track component of $x+1$. 

Now $H_\theta (e) (x-i)^{(1)} = c(x-i)^{(1)}$ for all $i \in \{0, \ldots, m \}$. Note that this also covers the cases when $\theta(x+1) = \gamma_n$ because the entire block of $\gamma$-cells was handled at once. Then by induction $H(e)(x) = c(x)$ for all $x < 0$. 

\begin{figure}
\begin{center}
\includegraphics[scale=0.8]{"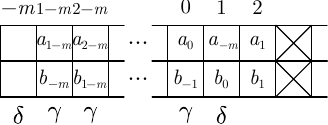"}
\caption{Construction of the pre-image going to the left.} \label{fig_surj3}
\end{center}
\end{figure}

Now $H_\theta (e)(x) = c(x)$ for all $x \in \Z$. Hence every $c \in \Sigma^\Z$ has a pre-image, meaning $H_\theta$ is surjective.
\qed
\end{proof}

\subsection{The general case}

Next, we reduce the general case of a non-recurrent rule distribution template to the special case from before. 


\begin{definition}
Let $T$ be a set of rule templates and $\tau \in T^\Z$ be a non-recurrent rule distribution template. Let $n\geq 1$ be such that there is a pattern $t_1 \ldots t_n$ in $\tau$ that appears only once. For every $\tau$ we treat one such unique pattern as fixed. Let $f$ and $g$ be local rules with state set $\Sigma$ and neighbourhood $N=(-r,\ldots,r)$ for some $r \in \Z_+$, $r\geq n$. 

The rule set $\mathcal{R}_{\tau,f,g}$ is such that for each $t\in T$ there is $h_t\in \mathcal{R}_{\tau,f,g}$ with state set $\Z_n \times \Sigma$ and neighbourhood $N$ which maps
\begin{align*}
((m_{-r}, a_{-r}),\ldots, (m_{r},a_{r})) \mapsto 
(m_0,f(a_{-r},\ldots, a_{r}))
\end{align*}
if $t=t_{m_0}$ and if in segment $(m_{-n},\ldots, m_n)$ there is a length $n+2$ sub-segment $(a,1,\ldots,n,b)$, where $a\neq n$ and $b\neq 1$. Otherwise $h_t$ maps
\begin{align*}
((m_{-r}, a_{-r}),\ldots, (m_{r},a_{r})) \mapsto 
(m_0,g(a_{-r},\ldots, a_{r})),
\end{align*}
for all $(m_i,a_i) \in \Z_n \times \Sigma$.
\end{definition}

Let's first clarify the function of the rule set $\mathcal{R}_{\tau,f,g}$. Configurations in $(\Z_n \times \Sigma)^\Z$ consist of two tracks, a static \emph{background track}, and an \emph{action track} that the local rules actually operate on. In the background track, the configuration makes a guess as to where in the word $t_1\ldots t_n$ the template in its cell is. Each rule in our assignment knows which template it is assigned to, so it can check if the guessed template is correct. If it is not, the rule $g$ is used on the action track. 

If the guess is correct, the rule further checks if the guess $m$ is part of the  sub-string $1\ldots n$ of a string $a1 \ldots n b$ on the background track, where $a\neq n$ and $b \neq 1$. If it is, the rule $f$ is used on the action track, and otherwise rule $g$ is used. Note that the local rule does not know whether the other guesses in this string are correct, only the guess in its own cell. An example is illustrated in Figure \ref{fig_guess}.

\begin{figure}
\begin{center}
\includegraphics[scale=0.8]{"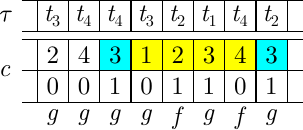"}
\caption{Example of the chosen rule set and assignment, where $n=4$. Which rules are used on the binary track are annoted on the bottom row. The string of guesses $1\ldots 4$ is highlighted. There are two correct guesses in this string and hence the rule $f$ is used in two cells.} \label{fig_guess}
\end{center}
\end{figure}

\begin{notation}
Let $\Sigma = A \times B$ for some state sets $A$ and $B$. For any $c \in \Sigma$, $c^{(1)} \in A^\Z$, $c^{(2)} \in B^\Z$ are such that $c^{(1)}(x) = c(x)^{(1)}$ and $c^{(2)}(x) = c(x)^{(2)}$.
\end{notation}


\begin{lemma} \label{gen}
Let $\tau \in T^\Z$ be a non-recurrent rule distribution template and $n\geq 1$ such that there is a length $n$ pattern $t_1 \ldots t_n$ in $\tau$ that appears only once. Let $\mathcal{R}_1 = \{f,g \}$ be a set of rules with alphabet $\Sigma$, $\mathcal{R}_2 = \mathcal{R}_{\tau,f,g}$ and $\alpha:T \rightarrow \mathcal{R}_2$ an assignment that maps $\alpha(t) = h_t$ for all $t \in T$. 

Suppose for all $\theta \in \mathcal{R}^\Z$ such that there is at most one length $n$ block of adjacent rules $f$ in $\theta$, the update rule $H_\theta$ is surjective or pre-injective. Then the update rule $H_{\tau_\alpha}$ is surjective or pre-injective, respectively.
\end{lemma}

\begin{proof}
Let  $B(c) = \set{e \in (\Z_n \times \Sigma)^\Z}{e^{(1)} = c^{(1)}}$ for any $c\in (\Z_n \times \Sigma)^\Z$ be the set of configurations that share the background track of $c$. For all $e \in B(c)$, rules $f$ and $g$ are used in the same cells, meaning there is a $\theta \in \mathcal{R}_1$ such that $H_{\tau_\alpha} (e)^{(2)} = H_\theta (e^{(2)})$ for all $e \in B(c)$. 

Furthermore, because pattern $t_1 \ldots t_n$ appears in $\tau$ only once, the pattern $h_{t_1}\ldots h_{t_n}$ appears only once in $\tau_\alpha$. Then for any $e \in B(c)$, there is at most one length $n$ block of cells where rule $f$ is used, because there is at most one length $n$ block of correct guesses for the templates in $e$ that are also a part of the substring $1\ldots n$ of $a 1 \ldots n b$, where $a \neq n$ and $b \neq 1$, in the background track. Therefore $\theta$ has at most one block of adjacent rules $f$.

Suppose then that for every such $\theta$, $H_\theta$ is surjective. Let $c_1 \in (\Z_n \times \Sigma)^\Z$ and $\theta \in \mathcal{R}_1$ such that $H_{\tau_\alpha} (e)^{(2)} = H_\theta (e^{(2)})$ for all $e \in B(c_1)$. Then because $H_\theta$ is surjective, there is $c_2 \in \Sigma^\Z$ such that $H_\theta (c_2) = c_1^{(2)}$. Let then $c_3 \in (\Z_n \times \Sigma)^\Z$ such that $c_3^{(1)} = c_1^{(1)}$ and $c_3^{(2)} = c_2^{(2)}$. Then $c_3 \in B(c_1)$ and hence $H_{\tau_\alpha} (c_3) = c_1$. Therefore $H_{\tau_\alpha}$ is surjective.

Suppose then that for every such $\theta$, $H_\theta$ is pre-injective. Let $c_1,c_2 \in (\Z_n \times \Sigma)^\Z$  be asymptotic configurations with $c_1 \neq c_2$. If $c_1^{(1)} \neq c_2^{(1)}$ then clearly $H_{\tau_\alpha} (c_1) \neq H_{\tau_\alpha} (c_2)$ because the background track is static. Assume then that $c_1^{(1)} = c_2^{(1)}$, meaning $c_1^{(2)} \neq c_2^{(2)}$. Then $B(c_1) = B(c_2)$, meaning there is $\theta \in \mathcal{R}_1^\Z$ such that $H_{\tau_\alpha} (c_1)^{(2)} = H_\theta (c_1^{(2)})$ and $H_{\tau_\alpha} (c_2)^{(2)} = H_\theta (c_2^{(2)})$. Now because $H_\theta$ is pre-injective, $H_{\tau_\alpha} (c_1)^{(2)} = H_\theta (c_1^{(2)}) \neq  H_\theta (c_2^{(2)}) = H_{\tau_\alpha} (c_2)^{(2)}$  and hence $H_{\tau_\alpha} (c_1) \neq H_{\tau_\alpha} (c_2)$. Therefore $H_{\tau_\alpha}$ is pre-injective.
\qed
\end{proof}

\begin{theorem}
Let $T$ be a set of rule templates and $\tau \in T^\Z$ be a non-recurrent rule distribution template. There exists a set of local rules $\mathcal{R}$ and assignment $\alpha: T \rightarrow \mathcal{R}$ such that $H_{\tau_\alpha}$ is pre-injective, but not surjective.
\end{theorem}

\begin{proof}
Let $n \geq 1$ be such that there is a length $n$ pattern $t_1\ldots t_n$ in $\tau$ that only appears once. Let $\mathcal{R}' = \{f_n, g_n\}$. By Lemma \ref{preinj} we know that for all $\theta \in \mathcal{R}'^\Z$ such that a length $n$ block of rules $f_n$ appears at most once, $H_\theta$ is pre-injective. Let $\mathcal{R} = \mathcal{R}_{\tau,f_n,g_n}$ and $\alpha:T\rightarrow \mathcal{R}$ such that $\alpha(t) = h_t$ for all $t \in T$. Then by Lemma \ref{gen}, $H_{\tau_\alpha}$ is pre-injective.

Next, assume the pattern $t_1\ldots t_n$ is in the cells $[x+1,x+n]$, meaning $\tau (x+i) = t_i$ when $1\leq i \leq n$. Let then $c \in (\Z_n \times \Sigma)^\Z$ be such that $c(x+i) = (i, 0)$ when $1\leq i \leq n$ and $c(x) = (m,1)$ for some $m\in \Z_n$. Now rule $f_n$ is used in cells $x+1,\ldots,x+n$ and rule $g_n$ is used in $x$. By Lemma \ref{nsurj}, $c$ then has no pre-image and therefore $H_{\tau_\alpha}$ is not surjective.
\qed
\end{proof}

\begin{theorem}
Let $T$ be a set of rule templates and $\tau \in T^\Z$ be a non-recurrent rule distribution template. There exists a set of local rules $\mathcal{R}$ and assignment $\alpha: T \rightarrow \mathcal{R}$ such that $H_{\tau_\alpha}$ is surjective, but not pre-injective. 
\end{theorem}

\begin{proof}
Let $n \geq 1$ be such that there is a length $n$ pattern $t_1\ldots t_n$ in $\tau$ that only appears once. Let $\mathcal{R}' = \{\gamma_n, \delta_n\}$. By Lemma \ref{surj} we know that for all $\theta \in \mathcal{R}'^\Z$ such that a length $n$ block of rules $\gamma_n$ appears at most once, $H_\theta$ is surjective. Let $\mathcal{R} = \mathcal{R}_{\tau,\gamma_n,\delta_n}$ and $\alpha:T\rightarrow \mathcal{R}$ such that $\alpha(t) = h_t$ for all $t \in T$. Then by Lemma \ref{gen}, $H_{\tau_\alpha}$ is surjective.

Next, assume the pattern $t_1\ldots t_n$ is in the cells $[x+1,x+n]$, meaning $\tau (x+i) = t_i$ when $1\leq i \leq n$. Let then $c \in (\Z_n \times \Sigma)^\Z$ be such that $c(x+i) = (i, 0)$ when $1\leq i \leq n$ and $c(x) = (m,1)$ for some $m\in \Z_n$. Now rule $\gamma_n$ is used in cells $x+1,\ldots,x+n$ and rule $\delta_n$ is used in $x$ and $x+n+1$. Then by Lemma \ref{npreinj} there are asymptotic differing configurations with the same image, meaning $H_{\tau_\alpha}$ is not pre-injective. 
\end{proof}

A similar reduction to the preceding can be achieved by using a technique found in \cite{salo}. Using this method, the assignment can be defined with a fixed number of states, not depending on the number of templates.

\section{Surjunctivity}

Surjunctivity is the property of a cellular automaton that if it's global update function is injective, it is surjective. Surjunctivity is implied by the Garden of Eden theorem, so it is known to hold for regular CA over integer grids. 

Furthermore, then we also know that surjunctivity holds for NUCA with a recurrent rule distribution. In this section we show a property of a rule distribution that implies surjunctivity, and show that every distribution asymptotic to a recurrent distribution has said property.

\begin{definition}
Let $k \in \Z$. The sets $(\infty,k)$ and $(k,\infty)$ is defined as 
\begin{align*}
(\infty, k) = \set{x \in \Z}{x < k}, \\
(k, \infty) = \set{x \in \Z}{x > k}.
\end{align*}
\end{definition}

\begin{lemma} \label{recur}

Let $\phi \in \mathcal{R}^\Z$ be a recurrent rule distribution. Let $i,j \in \Z$, $i \leq j$, and $u = \phi_{|(\infty,i)}$, $v = \phi_{|[i,j]}$, $w= \phi_{|(j,\infty)}$. Either every finite suffix of $u$ is a subword of $w$ or every finite prefix of $w$ is a subword of $u$.
\end{lemma}

\begin{proof}
Suppose $x$ is a suffix of $u$ and $y$ is a prefix of $w$ such that $x$ is not a subword of $w$ and $y$ is not a subword of $u$. Then the word $xvy$ cannot appear infinitely many times in $\phi$; if it did, either $x$ would be a subword of $w$ or $y$ a subword of $u$. Then $\phi$ isn't recurrent, which is a contradiction. Therefore either $x$ must be a subword of $w$ or $y$ must be a subword of $u$.
\qed
\end{proof}

We show that for any distribution $\theta$ with the property shown for recurrent distributions in Lemma \ref{recur}, the NUCA $H_\theta$ will be surjunctive. 

\begin{lemma} \label{surj}
Let $\theta \in \mathcal{R}^\Z$ be a rule distribution. Let $i,j\in \Z$ such that $i,j \in \Z$, $i \leq j$, $u = \theta_{|(\infty,i)}$, $w= \theta_{|(j,\infty)}$, and assume that either every finite suffix of $u$ is a subword of $w$ or every finite prefix of $w$ is a subword of $u$. Then if $H_\theta$ is injective, it is surjective. 
\end{lemma}

\begin{proof}
Assume $H_\theta$ is not surjective. Let $\Sigma$ be the state set of rules in $\mathcal{R}$ and assume that all rules are at most radius $r$. Assume that every suffix of $u$ is a subword of $w$.  The other case is identical. 

Let $u_n = \theta_{|[i-n,i-1]}$ be the length $n$ suffix of $\theta_{|(\infty,i)}$. For any $n$, let $m_n$ be the length of the segment from $i$ to the rightmost cell of the leftmost copy of $u_n$ to the right of $j$. Let then $\psi_n \in \mathcal{R}^{\Z_{m_n}}$ be such that $\psi_n (x \mod m_n) = \theta (x)$ for all $x \in [i, i+m_n-1]$. We can think of $\psi_n$ as the segment of $\theta$ ranging from $u_n$ to a copy of $u_n$, wrapped around in a circle with the two copies overlapping each other. This is illustrated in Figure \ref{wrap1}.

\begin{figure}
\begin{center}
\includegraphics[scale=0.8]{"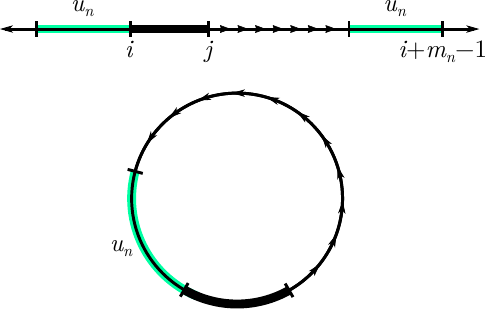"}
\caption{The wrapping of $\theta$ (above) to $\psi_n$ (below). The arrows indicate the direction of the wrapping.} \label{wrap1}
\end{center}
\end{figure}

Because $H_\theta$ is not surjective, for large enough $n$, the function $H_{\psi_n}$ is not surjective. Since $\Sigma^{\Z_{m_n}}$ is finite, $H_{\psi_n}$ is not injective. Let then $c_n,e_n \in \Sigma^{\Z_{m_n}}$ be configurations such that $c_n \neq e_n$ and $H_{\psi_n} (c_n) = H_{\psi_n} (e_n)$. Now there are two cases.

\emph{Case} 1: For infinitely many $n$, the $c_n$ and $e_n$ differ within $2r$  of the segment $[i,j]$, that is, $\mathrm{diff}(c_n, e_n) \cap [i-2r,j+2r] \neq \emptyset$. Let then $c_n', e_n' \in \Sigma^\Z$ be configurations such that $c_n' (x) = c_n(x \mod m_n)$ and $e_n' (x) = e_n(x \mod m_n)$ for all $x\in [i-n,i+m_n-1]$ and $c_n'(y) = e_n'(y)$ for all $y \notin [i-n,i+m_n-1]$ The configurations $c_n'$ and $e_n'$ can be thought of as $c_n$ and $e_n$ "unwrapped" and "embedded" into some configuration over $\Z$.

Let then $(c_{n_k}',e_{n_k}')_k$ be a sequence of pairs of such configurations. By compactness, this sequence has a converging subsequence with a limit $(c,e)$. Now $c \neq e$, because $\mathrm{diff}(c_{n_k}', e_{n_k}') \cap [i-2r,j+2r] \neq \emptyset$ for all $k$. In addition, because $H_{\psi_{n_k}} (c_{n_k}) = H_{\psi_{n_k}} (e_{n_k})$ for all $k$, for any finite domain $D \subset \Z$ there is $m \in \Z_+$ such that $H_\theta(c_{n_k}')(D)=H_\theta(e_{n_k}')(D)$ for all $k \geq m$. Therefore $H_\theta (c) = H_\theta (e)$, meaning $H_\theta$ is not injective.

\emph{Case} 2: For all large enough $n$, $c_n$ and $e_n$ are identical within $2r$ of $[i,j]$. Let $c,e\in \Sigma^\Z$ such that for some such $n$, $c(x) = c_n(x \mod m_n)$ and $c(x) = c_n(x \mod m_n)$ for all $x \in [j+1, i+m_n-1]$ and $c(y) = e(y)$ for all $y \notin [j+1, i+m_n-1]$. This is illustrated in Figure \ref{wrap2}.

\begin{figure}
\begin{center}
\includegraphics[scale=0.8]{"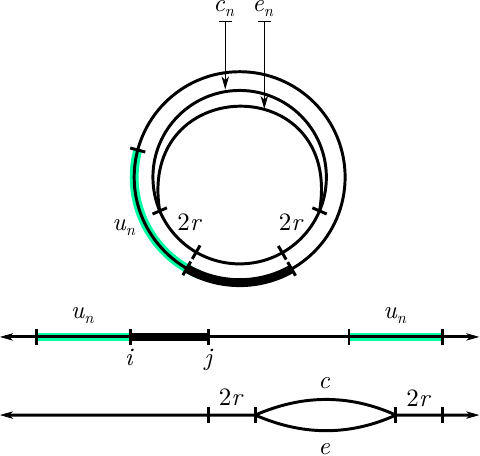"}
\caption{The unwrapping of $c_n$ and $e_n$ into $c$ and $e$.} \label{wrap2}
\end{center}
\end{figure}

Clearly $c \neq e$, because $c_n$ and $e_n$ differ somewhere in the segment $[j+1, i+m_n-1]$. For any cell $x$ that is at least $r$ cells away from $\mathrm{diff}(c,e)$, the neighbourhood of $x$ is identical in $c$ and $e$, hence $H_\theta (c)(x) = H_\theta(e)(x)$. For any cell $y$ that is within $r$ cells of $\mathrm{diff}(c,e)$, its neighbourhood is within $2r$ of $\mathrm{diff}(c,e)$. Then because $H_{\psi_n} (c_n) = H_{\psi_n} (e_n)$, $H_\theta (c)(y) = H_\theta(e)(y)$. Therefore $H_\theta (c) = H_\theta (e)$, meaning $H_\theta$ is not injective.

Hence in either case, $H_\theta$ is not injective. Therefore if $H_\theta$ is injective, it is surjective.
\qed
\end{proof}

\begin{theorem}
Let $\theta \in \mathcal{R}^\Z$ be asymptotic to a recurrent rule distribution $\phi \in \mathcal{R}^\Z$. If $H_\theta$ is injective, it is surjective.
\end{theorem}

\begin{proof}
Let $i,j\in \Z$, $i\leq j$ be such that $\mathrm{diff}(\theta, \phi) \subseteq [i,j]$ and let $u = \phi_{|(\infty,i)}$ and $w= \phi_{|(j,\infty)}$. By Lemma \ref{recur}, either every finite suffix of $u$ is a subword of $w$ or every finite prefix of $w$ is a subword of $u$. Then because $\phi$ and $\theta$ are identical outside of $[i,j]$, by Lemma \ref{surj}, if $H_\theta$ is injective, it is surjective.
\qed
\end{proof}

\section{Conclusions}

We find that the Garden of Eden theorem holds for NUCA if the local rule distribution is uniformly recurrent. In the 1-dimensional case we find that every assignment to a given rule template defines a NUCA that satisfies either direction of the Garden of Eden theorem, if and only if the template is recurrent. Finally we find that all rule distributions asymptotic to a recurrent distribution are surjunctive.

The Garden of Eden theorem for NUCA should still be examined in other groups. As for surjunctivity, we have shown a property of a template that guarantees surjunctivity, but know nothing about the converse. It may be useful to examine the complement of the underlying property which gives us surjunctivity, and see whether this guarantees the existence of non-surjunctive assignments.

\bibliographystyle{splncs04}
\bibliography{nonunifgoe}

%
%
%
%
%
%
%
%
%
%
%
\end{document}